\documentclass[a4paper,twoside,reqno]{amsart}  

\usepackage[english]{babel}
\usepackage[utf8]{inputenc}

\usepackage[sc, osf]{mathpazo}
\usepackage{enumerate}
\usepackage{graphicx}
\usepackage{color}

\usepackage{amsmath, amsthm, amssymb, amsfonts}
\numberwithin{equation}{section}

\newtheorem*{theorem*}{Theorem}
\newtheorem{theorem}{Theorem}[section]
\newtheorem{proposition}{Proposition}[section]
\newtheorem{corollary}{Corollary}[theorem]
\newtheorem{lemma}{Lemma}[section]

\theoremstyle{remark}
  \newtheorem{remark}{Remark}[section]

\usepackage[unicode,psdextra]{hyperref}
\hypersetup{
  pdftitle={Rotationally invariant constant Gauss curvature surfaces in the Berger spheres},%
  pdfauthor={Francisco Torralbo and Joeri Van der Veken},%
  pdfkeywords={},%
  pdfcreator={pdfLaTeX},%
  pdfproducer={LaTeX with hyperref},
  urlcolor=[rgb]{0.2,0.1,0.5},
  linkcolor=[rgb]{0.2,0.2,0.2},
  citecolor=[rgb]{0.2,0.2,0.2}
}

\title{Rotationally invariant constant Gauss curvature surfaces in Berger spheres}
\author{\href{http://www.ugr.es/local/ftorralbo/}{Francisco Torralbo}}
\address{Dpto.\ de Geometr\'{\i}a y Topolog\'{\i}a - Facultad de Ciencias \\
  Universidad de Granada \\
  Fuentenueva s/n 18071 Granada (Spain)}
\email{ftorralbo@ugr.es}
\author{\href{https://perswww.kuleuven.be/~u0043959/indexeng.html}{Joeri Van der Veken}}
\address{KU Leuven \\
  Department of Mathematics \\
  Celestijnenlaan 200B - Box 2400 \\
  3001 Leuven (Belgium).
}
\email{joeri.vanderveken@kuleuven.be}

\thanks{The first author is supported by the Spanish \textsc{mineco-feder} research project \textsc{mtm}2017-89677-\textsc{p} and the research project \textsc{p9-ugr-curvature}. The second author is supported by the Excellence Of Science project \textsc{g0h4518n} of the Belgian government and project \textsc{3e160361} of the KU Leuven Research Fund.}



\begin{document}

\maketitle

\begin{abstract}\small
  We give a full classification of complete rotationally invariant surfaces with constant Gauss curvature in Berger spheres: they are either Clifford tori, which are flat, or spheres of Gauss curvature $K \geq K_0$ for a positive constant $K_0$, which we determine explicitly and depends on the geometry of the ambient Berger sphere. For values of $K_0 \leq K \leq K_P$, for a specific constant $K_P$, it was not known until  now whether complete constant Gauss curvature $K$ surfaces existed in Berger spheres, so our classification provides the first examples. For $K > K_P$, we prove that the rotationally invariant spheres from our classification are the only topological spheres with constant Gauss curvature in Berger spheres. 
\end{abstract}

\section{Introduction}
\label{sec:introduction}

The classification of surfaces with constant Gauss curvature (\textsc{cgc} in the sequel) in the $3$-sphere is a classical topic in submanifold theory. Let $\mathbb{S}^3$ represent the $3$-dimensional sphere of constant sectional curvature $1$. If $S$ is a complete \textsc{cgc} $K$ surface in $\mathbb S^3$, then either $K \geq 1$ and $S$ is a totally umbilical sphere, or $K = 0$. In particular, there are no complete \textsc{cgc} $K$ surfaces with $0 < K < 1$, nor with $K < 0$. Moreover, for $K = 0$ there is a rich and classical theory due to Bianchi: flat surfaces are constructed by multiplying two intersecting asymptotic curves using the unit quaternion group structure (see~\cite{Galvez2009} and the references therein). 

Our main goal is to study \textsc{cgc} surfaces in the Berger spheres. These Riemannian $3$-spheres are the most symmetric ones after the round sphere and they are obtained by deforming the metric of a round sphere in the direction of the fibers of the Hopf fibration. Up to a homothecy, they can be represented as a $1$-parameter family $\{\mathbb{S}^3_\tau \, : \, \tau > 0\}$, where the fibers of $\mathbb S^3_{\tau}$ are enlarged by a factor $\tau$ (see Section~\ref{sec:preliminaries} for more details). Hence, the round sphere is the element of the family with $\tau = 1$, i.e., $\mathbb S^3_1 = \mathbb S^3$.

On the one hand, if we restrict ourselves to the study of \textsc{cgc} spheres, we can apply a general existence result due to Pogorelov~\cite[Theorem 1, p.~413]{Pogorelov1973} (see also~\cite[Theorem B, p.~419]{Labourie1989}) that ensures the existence of an isometric immersion of a sphere into a Riemannian $3$-manifold if its Gauss curvature $K$ is everywhere strictly larger than an upper bound of the sectional curvature of the manifold. This condition ensures that the determinant of the shape operator is strictly positive and the result uses the theory of strongly elliptic \textsc{pde}s. The sectional curvature of a plane $\Pi$ with unit normal $N$ in the Berger sphere $\mathbb S^3_{\tau}$ is given by
\begin{equation}\label{eq:sectional-curvature-homogeneous}
  K(\Pi) = \tau^2 + 4(1 - \tau^2)g_\tau(N, \xi)^2,
\end{equation}
where $g_\tau$ is the Berger metric and $\xi$ is a unit vector field tangent to the fibers of the Hopf fibration. Hence, we immediately get the following existence result:
\begin{quote}\itshape
  Let $\mathbb{S}^3_\tau$ be a Berger sphere and $K$ a positive constant. If 
\begin{equation}\label{eq:bound-sectional-curvature-homogeneous}
  K > K_P = 
  \begin{cases}
      4 - 3 \tau^2 & \text{if } \tau \leq 1, \\
      \tau^2 & \text{if } \tau > 1, \\
  \end{cases}
\end{equation}
  then there exists a constant Gauss curvature $K$ sphere immersed in $\mathbb{S}^3_\tau$.
\end{quote}
Moreover, Pogorelov's result also ensures the uniqueness of the sphere once a point and an exterior normal to the sphere at that point are fixed (see the striped region in Figure~\ref{fig:non-existence-region-compact-cgc-surfaces-Berger}).

On the other hand, Urbano and the first author~\cite{TU2010} gave non-existence intervals for the Gauss curvature of compact \textsc{cgc} surfaces in the Berger spheres (see light gray region in Figure~\ref{fig:non-existence-region-compact-cgc-surfaces-Berger}). They showed that there are no compact  \textsc{cgc} surfaces for $K \le 0$ unless  $K = 0$ and the surface is a Hopf torus (i.e., the preimage of a closed curve by the Hopf fibration) or $K = 4(1 - \tau^2) < 0$ for which no examples are known. Moreover, they also show that there are no complete \textsc{cgc} surfaces for $0 < K < K_{TU}$ for a positive constant $K_{TU}$ depending on $\tau$ that turns out to be strictly less than $K_P$ (see~\eqref{eq:bound-sectional-curvature-homogeneous} and light gray region in Figure~\ref{fig:non-existence-region-compact-cgc-surfaces-Berger}).

In this paper we make progress on the existence problem of complete \textsc{cgc} surfaces by classifying the examples which are invariant by a $1$-parameter group of ambient isometries. In particular, we classify the complete \textsc{cgc} surfaces in Berger spheres that are \emph{rotationally invariant} (see Section~\ref{sec:rotationally-invariant-examples} for more details).

\begin{theorem}\label{thm:main-theorem}
Consider the Berger sphere $\mathbb{S}^3_\tau$ and let $K$ be a constant such that
\begin{equation}\label{eq:lower-bound-Gauss-curvature-rotational-examples}
   K \geq K_0 = 
    \begin{cases}
      4 - 3 \tau^2 & \text{if } \tau \leq 1, \\
      \frac{1}{\tau^2} & \text{if } \tau > 1. \\
    \end{cases}
  \end{equation}
  Then, up to ambient isometries, there exists a unique rotationally invariant sphere with constant Gauss curvature $K$ in $\mathbb{S}^3_\tau$. Together with the Clifford tori, which are flat, these surfaces are the only complete rotationally invariant \textsc{cgc} surfaces in $\mathbb{S}^3_\tau$.
\end{theorem}

\begin{figure}[htb]
  \centering
  \includegraphics[width=11cm]{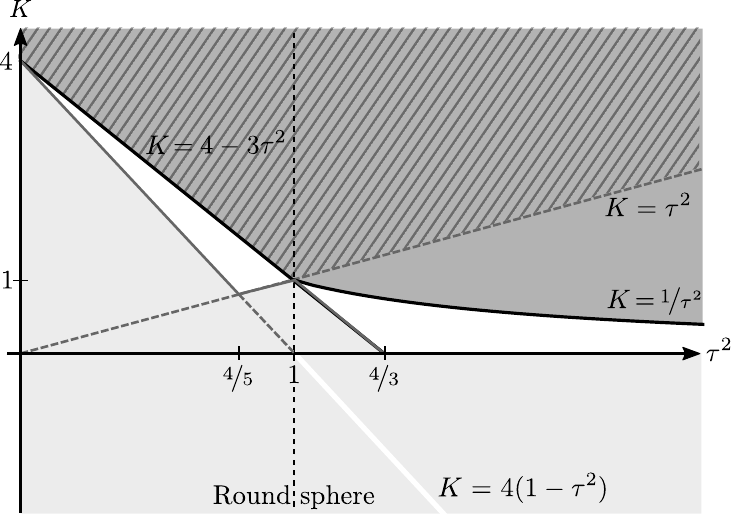}
  \caption{Each point represents a compact \textsc{cgc} $K$ surface in the Berger sphere $\mathbb{S}^3_\tau$ ($\tau^2$ in the horizontal axis and $K$ is on the vertical axis). The non-existence region for compact \textsc{cgc} surfaces given in~\cite{TU2010} is colored in light gray (excluding the case $K = 0$ that corresponds to Hopf tori), while the existence region for rotationally invariant \textsc{cgc} spheres is colored in dark gray, see Proposition~\ref{prop:CGC-spheres}. The surfaces in the striped area are unique (hence they are the rotationally invariant examples, see Corollary~\ref{cor:uniqueness}) but the uniqueness is unknown in the boundary of that region. The existence of compact \textsc{cgc} surfaces is unknown in the white area.}
  \label{fig:non-existence-region-compact-cgc-surfaces-Berger}
\end{figure}

Recall that a Clifford torus is the preimage of a circle by the Hopf fibration. It is interesting to remark that when the length of the fiber (which is proportional to $\tau$) is small enough, the \textsc{cgc} spheres are not embedded (see Figures~\ref{fig:region-non-embeddedness}).

Theorem~\ref{thm:main-theorem} improves the lower bound $K_P$ for the Gauss curvature  coming from Pogorelov's existence theorem because $K_0 \leq K_P$ (compare~\eqref{eq:bound-sectional-curvature-homogeneous} and \eqref{eq:lower-bound-Gauss-curvature-rotational-examples}). More precisely, there exist \textsc{cgc} $K$ spheres for $K = 4 -3\tau^2$ (if $\tau \leq 1$) and for $\frac{1}{\tau^2} \leq K \leq \tau^2$ (if $\tau > 1$). Observe that, if $K_0 \leq K \leq K_P$ then the determinant of the shape operator vanishes at some points of the sphere so Pogorelov's proof fails since the theory of strongly elliptic \textsc{pde}s can not longer be applied. Figure~\ref{fig:non-existence-region-compact-cgc-surfaces-Berger} summarizes the result. 

Moreover, from a detailed analysis of the profile curve of the rotationally invariant spheres together with Pogorelov's uniqueness result, we get the following corollary (see the striped region in Figure~\ref{fig:non-existence-region-compact-cgc-surfaces-Berger}).

\begin{corollary}\label{cor:uniqueness}
  If $S$ is a constant Gauss curvature $K$ sphere in  $\mathbb{S}^3_\tau$ such that $K > 4 - 3 \tau^2$ if  $\tau^2 \leq 1$ or  $K > \tau^2$ if  $\tau^2 > 1$ then it is rotationally invariant.
\end{corollary}

Finally, the Berger spheres are examples of homogeneous Riemannian $3$-ma\-ni\-folds with isometry group of dimension $4$. This family is classified in terms of two real parameters $\kappa$ and $\tau$ and they are usually named $E(\kappa, \tau)$-spaces (see, for instance, \cite{DHM2009}). In particular, $\mathbb{S}^3_\tau = E(4, \tau)$. They are characterized (see~\cite{Man2014, Joeri2008}) as simply connected $3$-manifolds with a Riemannian submersion $\pi: E(\kappa, \tau) \to \mathbb{M}^2(\kappa)$, where $\mathbb{M}^2(\kappa)$ stands for a simply connected complete constant curvature $\kappa$ surface, with constant bundle curvature $\tau$ and whose fiber is generated by a unit Killing vector field $\xi$. If $\kappa - 4\tau^2 = 0$, then we get a space form. In general, the sectional curvature of a plane $\Pi$ with unit normal $N$ is given by $K(\Pi) = \tau^2 + (\kappa - 4\tau^2)\langle N, \xi\rangle^2$, where $\langle \cdot, \cdot \rangle$ is the metric in  $E(\kappa, \tau)$ (cp.~\eqref{eq:sectional-curvature-homogeneous}), so a similar lower bound for the existence of \textsc{cgc} $K$ surfaces as in~\eqref{eq:bound-sectional-curvature-homogeneous} can be written down for $E(\kappa, \tau)$. Different studies concerning \textsc{cgc} surfaces in $E(\kappa, \tau)$ spaces have been done.

For example, all surfaces in $E(\kappa,\tau)$ for which the unit normal makes a constant angle with the unit Killing vector field $\xi$ have constant Gaussian curvature, see~\cite{DFV07, DM09, FMV09, MO2014, MOP2016} for classifications of this relatively large family of examples in the different types of $E(\kappa,\tau)$ spaces.

In the rotationally invariant setting, \textsc{cgc} surfaces in the Heisenberg space $\mathrm{Nil}_3 = E(0, \frac{1}{2})$ were described in~\cite{CPR96} (see also~\cite{MO2005} for a general description of \textsc{cgc} surfaces invariant by a $1$-parameter group of isometries). Among them, only those with $K \geq 0$ are complete. Moreover, in the product spaces $\mathbb{S}^2\times \mathbb{R} = E(1,0)$ and $\mathbb{H}^2\times \mathbb{R} = E(-1,0)$, the complete ones were described in~\cite{AEG07-1} where the authors showed that they exist for $K \geq 1$ and $K = 0$ in $\mathbb{S}^2 \times \mathbb{R}$ and for $K \ge -1$ in $\mathbb{H}^2\times \mathbb{R}$. More importantly, the authors obtained Liebmann and Hilbert type results in those spaces: there is a unique complete \textsc{cgc} surface for $K > 1$ in $\mathbb{S}^2\times \mathbb{R}$ and for $K > 0$ in $\mathbb{H}^2 \times \mathbb{R}$ (that is the rotationally invariant sphere), and there is no complete \textsc{cgc} surfaces for $K < -1$ in $\mathbb{S}^2\times \mathbb{R}$ and $\mathbb{H}^2\times \mathbb{R}$. However, it is unknown whether there are complete examples for $-1 \le K < 0$ and $0 < K < 1$ in $\mathbb{S}^2\times \mathbb{R}$ and it is not known whether the rotationally invariant examples for $K = 1$ in $\mathbb{S}^2 \times \mathbb{R}$ and $-1 \le K \le 0$ in $\mathbb{H}^2 \times \mathbb{R}$ are unique. This is similar to the case $\tau < 1$ in the Berger sphere $\mathbb{S}^3_\tau$ where the uniqueness remains unknown only for $K = 4 - 3\tau^2$. 

The authors would like to thanks to J.\ A.\ Gálvez and F.\ Urbano for useful conversations during the preparation of this manuscript.

\section{Preliminaries}\label{sec:preliminaries}

The Berger spheres form a family of homogeneous Riemannian $3$-manifolds, diffeomorphic to $3$-spheres, which have isometry groups of dimension $4$. Although they are described in the literature as a $2$-parameter family (e.g.~\cite{Torralbo2012}), we can describe them, up to homothecy, as a $1$-parameter family as follows: a Berger sphere $\mathbb{S}^3_\tau$ is the $3$-sphere $\{(z, w) \in \mathbb{C}^2\colon |z|^2 + |w|^2 = 1\}$ endowed with the metric
\[
  g_{\tau}(X, Y) = \langle X, Y \rangle - ( 1 - \tau^2) \langle X, V\rangle \langle Y, V \rangle,
\] 
where $\langle \, , \rangle$ stands for the metric on the sphere induced from the Euclidean metric on $\mathbb C^2 = \mathbb R^4$, and the vector field $V$ is given by $V_{(z,w)} = (iz, iw)$. If  $\tau = 1$, then the metric  $g_{1}$ is the round metric of curvature $1$ on the $3$-sphere. 

The usual Hopf fibration, 
\[
\Pi: \mathbb{S}^3_\tau \to \mathbb{S}^2(4): (z, w) \mapsto \bigl( z \bar{w}, \, \tfrac{1}{2}(|z|^2 - |w|^2) \bigr),
\] 
where $\mathbb{S}^2(4)$ denotes the round $2$-sphere of curvature $4$ in $\mathbb R^3 = \mathbb C \times \mathbb R$, is a Riemannian submersion. The fibers of $\Pi$ are geodesic circles of length $2\pi \tau$ and $\xi =  \frac{1}{\tau}V$ is a unit Killing vector field tangent to these fibers. Remark that, if $\tau$ goes to zero, the length of the vertical geodesics also goes to zero and, in the limit, the $3$-sphere degenerates to a $2$-sphere.

We can group the non-round Berger spheres in two classes: those with $\tau < 1$ can be realized as geodesic spheres in the complex projective plane, while those with  $\tau > 1$ can be seen as geodesic spheres in the complex hyperbolic plane (see~\cite{TU2011} for more details). This sheds some light on the different behavior in our main results according to whether $\tau \geq 1$ or $\tau < 1$. In the following, we will introduce the auxiliary parameter
\begin{equation} \label{eq:lambda}
\lambda = 1 -\tau^2
\end{equation}
to simplify the notation. Clearly, $\lambda < 1$ always. Moreover, $1>\lambda>0$ corresponds to $\tau < 1$; $\lambda=0$ corresponds to the round case $\tau = 1$ and $\lambda<0$ corresponds to $\tau > 1$.

\section{The rotationally invariant examples in the Berger sphere}
\label{sec:rotationally-invariant-examples}

This section is devoted to finding all the complete rotationally invariant \textsc{cgc} surfaces in Berger spheres. We call a surface \emph{rotationally invariant} if it is invariant by the $1$-parameter group of isometries of the Berger sphere given by
\[
\mathrm{Rot} = \left\{ \begin{pmatrix}1 & 0 \\  0 & e^{it}\end{pmatrix}:\, t \in \mathbb{R}\right\}.
\]
This group fixes the curve $\ell = \{(z, 0) \in \mathbb{S}^3_\tau\}$, which we will call the \emph{axis of revolution}. Any other $1$-parameter group of isometries fixing a curve is actually conjugate to $\mathrm{Rot}$ (see \cite[Proposition 3]{Torralbo2010a}), so there is no geometric restriction in our choice. 

Observe that the set $\mathbb{S}^3_\tau/\mathrm{Rot}$ can be identified with the closed hemisphere $H = \{(z,a) \in \mathbb{C} \times \mathbb{R}:\, |z|^2 + a^2 = 1, \ a \geq 0 \} \subset \mathbb{S}^2(1)$, bounded by $\ell$. Hence, any rotationally invariant surface is the orbit by $\mathrm{Rot}$ of a curve 
\[
\gamma: I \subset \mathbb{R} \to H \subset \mathbb S^2(1) : s \mapsto \left(e^{iy(s)}\cos x(s), \, \sin x(s) \right), 
\]
with $\sin x(s) \geq 0$, which we will call the \emph{profile curve of the surface}. Observe that if $\sin x(s) = 0$ then the curve $\gamma$ intersects the axis of revolution and a singularity in the surface might appear. 

The surface obtained by moving $\gamma$ by $\mathrm{Rot}$ can be parametrized by
\[
\Phi: I \times \mathbb{R} \to \mathbb{S}^3_\tau: (s,t) \mapsto \left(e^{iy(s)} \cos x(s), e^{it} \sin x(s) \right).
\]
The components of the first fundamental form are
\begin{equation} \label{eq:components_1st_fund}
  \begin{split}
    E(s) &= g_{\tau}(\Phi_s, \Phi_s) = x'(s)^2 + \cos^2 x(s) (1-\lambda \cos^2 x(s))y'(s)^2, \\ 
    F(s) &= g_{\tau}(\Phi_s, \Phi_t) = -\lambda \sin^2 x(s) \cos^2 x(s) y'(s), \\ 
    G(s) &= g_{\tau}(\Phi_t, \Phi_t) = \bigl(1- \lambda\sin^2 x(s)\bigr)\sin^2 x(s),
  \end{split}
\end{equation} 
where $\lambda = 1 - \tau^2$ (see~\eqref{eq:lambda}).

Observe that $G(s_0) = 0$ if and only if $\sin x(s_0) = 0$, that is, the curve $\gamma$ touches the axis of revolution $\ell$ in the point $\gamma(s_0)$. On the other hand, if $x(s_0) = \frac{\pi}{2}$, then $E(s_0) = x'(s_0)$, so $\Phi(s_0,t)$ is a regular curve if and only if $x'(s_0)\neq 0$. In particular, there is no surface of revolution for which $x(s)$ is constant $\frac{\pi}{2}$ (the profile curve $\gamma$ would degenerate to a point).

\begin{lemma}\label{lm:ode}
  Let $\gamma: I \subset \mathbb{R} \to H \subset \mathbb S^2(1): s \mapsto \bigl(e^{iy(s)}\cos x(s) , \sin x(s) \bigr)$ be the profile curve of a rotationally invariant surface with Gauss curvature~$K$ in $\mathbb{S}^3_\tau$. After a suitable reparametrization, there exists a function $\alpha: I \to \mathbb{R}$ such that $x$, $y$ and $\alpha$ satisfy the following system of ordinary differential equations:
  \begin{equation}\label{eq:ODE-Rotationally-invariant-surface}
  \begin{split}
    x' &= \cos \alpha, \\
    y' &= \frac{1}{\tau} \frac{\sqrt{1 \! - \! \lambda\sin^2 x}}{\cos x} \sin \alpha, \\
    \alpha' &= \frac{\tan x}{\sin \alpha}\left[ \frac{1 - \lambda\sin^2 x}{1 - 2 \lambda\sin^2 x}K - \cos^2 \alpha \left( \frac{1 - \lambda}{1  - \lambda \sin^2 x} + \frac{4 \lambda \cos^2 x}{1  - 2 \lambda\sin^2 x} \right)\right],
  \end{split}
\end{equation}
where $\lambda = 1 - \tau^2$ (see~\eqref{eq:lambda}). 

Moreover, if $K$ is constant and $(x(s),y(s),\alpha(s))$ is a solution of \eqref{eq:ODE-Rotationally-invariant-surface}, the quantity
\begin{equation}\label{eq:integral-formula}
    \mathcal E = \frac{(1 - 2\lambda \sin^2 x)^2}{1 \! - \! \lambda \sin^2 x} \cos^2 x \cos^2 \alpha + K(1 - \lambda \sin^2 x)\sin^2 x
\end{equation}
is also constant and we call it the \emph{energy} of the solution.
\end{lemma}

\begin{proof}
We parametrize the profile curve $\gamma: I \to H$ such that
\[
x'(s)^2 + \frac{(1- \lambda)\cos^2 x(s)}{1- \lambda \sin^2x(s)}y'(s)^2 = 1.
\] 
As a consequence, there exists a smooth function $\alpha: I \to \mathbb{R}$ such that the first two equations of~\eqref{eq:ODE-Rotationally-invariant-surface} are satisfied. Moreover, a straightforward computation using \eqref{eq:components_1st_fund} shows that $E(s)G(s) - F(s)^2 = G(s)$ and, if we define $\varphi(s) = \sqrt{G(s)}$, the formula of Frobenius (expressing the Gauss curvature in terms of the components of the first fundamental form and their derivatives) gives $\varphi''(s) + K \varphi(s) = 0$, which implies the third equation of \eqref{eq:ODE-Rotationally-invariant-surface}.

Now, if $K$ is constant, the first integral of $\varphi''(s) + K \varphi(s) = 0$ is precisely $\varphi'(s)^2 + K \varphi(s)^2 = \mathcal E$ for some constant $\mathcal E$. This is equivalent to~\eqref{eq:integral-formula}. 
\end{proof}

\begin{remark}\label{rmk:symmetries-ODE-system}
  In the round sphere case, that is, for $\lambda = 0$, the system~\eqref{eq:ODE-Rotationally-invariant-surface} was studied by Hsiang~\cite{Hsiang1983}. We now describe the symmetries of the system \eqref{eq:ODE-Rotationally-invariant-surface} in general. Let $(x,y,\alpha)$ be a solution, $I$ its maximal interval of definition, and $s_0 \in I$.
\begin{enumerate}[(i)]

  \item Translations along the $y$-axis produce new solutions, that is, $(x, y + y_0, \alpha)$ is a solution of~\eqref{eq:ODE-Rotationally-invariant-surface} for any constant $y_0$. This corresponds to \emph{vertical} translations (the $1$-parameter group of diffeomorphisms generated by the Killing field $\xi$) of the profile curve in Berger spheres, which are isometries for any $\tau$. 

  \item The function $\alpha$ is determined up to an integer multiple of $2\pi$, that is, $(x, y, \alpha + 2k\pi)$ is a solution of~\eqref{eq:ODE-Rotationally-invariant-surface} for any constant $k\in\mathbb Z$.

  \item Reversal of the parameter produces a solution, that is, $\bigl(x(2s_0 - s), y(2s_0 - s), \alpha(2s_0 - s) + \pi\bigr)$ is a solution of~\eqref{eq:ODE-Rotationally-invariant-surface} defined in a neighborhood of $s_0$.

  \item Reflection in a line $y = y_0$ also produces a solution, that is, $(x, 2y_0 - y, -\alpha)$, is a solution of \eqref{eq:ODE-Rotationally-invariant-surface} for any constant $y_0$.

  \item If $x'(s_0) = 0$ then the solution is symmetric with respect the line $y = y(s_0)$. More precisely, the condition $x'(s_0) = 0$ implies by \eqref{eq:ODE-Rotationally-invariant-surface} that $\alpha(s_0) \in \{-\frac{\pi}{2}, \frac{\pi}{2}\}$. Hence, the two solutions $(x,y,\alpha)$ and $(\hat{x}, \hat{y}, \hat{\alpha})$ where
    \[
      \begin{split}
        \hat{x}(s) &= x(2s_0 - s), \quad \hat{y}(s) = 2y(s_0) - y(2s_0 - s), \\ 
        \hat{\alpha}(s) &= 
        \begin{cases}
          \pi - \alpha(2s_0 - s) & \text{if } \alpha(s_0) = \tfrac\pi2,\\
          -\pi - \alpha(2s_0 - s) & \text{if } \alpha(s_0) = -\tfrac\pi2,
        \end{cases}
      \end{split}
    \]
    coincide at $s_0$ so they are equal.

  \item Finally, if $\sin x(s_0) = 1$, that is, the profile curve $\gamma$ contains the north pole of the hemisphere $H \subset \mathbb S^2(1)$, then we can smoothly continue $\gamma$ by considering $(x, y + \pi, \alpha)$. This corresponds to a rotation over an angle $\pi$ of $\gamma$ around the north pole of $H$ or, equivalently, to a translation of $\pi$ along the fibers in the Berger sphere.

\end{enumerate}

From the above we can always assume that $y_0 = 0$ and that for any $(x_0, \alpha_0)\in [0, \frac{\pi}{2}[ \times [0, \pi[$ there exists a solution of~\eqref{eq:ODE-Rotationally-invariant-surface} with constant Gauss curvature $K$ and energy $\mathcal E$ given by~\eqref{eq:integral-formula}. Moreover, the corresponding rotationally invariant surface is complete if $\gamma$ is defined for all $s$ or if the pair $(\sin^2 x, \cos \alpha)$ has either $(0, \pm1)$ or $(1, a)$, for any $a \in [-1,1]$, as limit values. In the first case the profile curve $\gamma$ touches the axis of revolution orthogonally, while in the second case the curve contains the north pole and it can be continued as in (vi).
\end{remark}

\begin{lemma}\label{lm:constant-solutions}
The only solutions of~\eqref{eq:ODE-Rotationally-invariant-surface} for which $x$,  $y$ or  $\alpha$ is constant are
\begin{enumerate}[(i)] 
\item solutions of the form 
$$ (x(s),y(s),\alpha(s)) = \left( x_0, \ \frac{1}{\tau} \frac{\sqrt{1-\lambda \sin^2 x_0}}{\cos x_0}s, \ \frac{\pi}{2} \right) $$ 
for some constant $x_0$ which is not an integer multiple of $\frac{\pi}{2}$;
\item in the case of a round sphere, solutions of the form
$$ (x(s),y(s),\alpha(s)) = (s, y_0, 0) $$
for some constant $y_0$.
\end{enumerate}
Solutions of type (i) correspond to \emph{Clifford tori}, i.e., preimages of constant curvature curves under the Hopf fibration, which are flat. Solutions of type (ii) only appear for round spheres, i.e., when $\tau = 1$, and correspond to totally geodesic $2$-spheres.
\end{lemma}

\begin{proof}
If $x$ or $y$ are constant, the first or second equation of \eqref{eq:ODE-Rotationally-invariant-surface} shows that $\alpha$ is also constant. So we only have to investigate the case where $\alpha$ is constant.

If $\alpha(s)=\alpha_0$, the first equation of~\eqref{eq:ODE-Rotationally-invariant-surface} gives $x(s) = (\cos \alpha_0) s + x_0$ and from the third equation of~\eqref{eq:ODE-Rotationally-invariant-surface} we deduce that 
  \[
    K = \frac{\cos^2 \alpha_0}{(1 - \lambda \sin^2 x(s))^2}
    \left[ 4\lambda^2 \sin^4 x(s) - 2\lambda(\lambda + 3) \sin^2 x(s) + 3\lambda + 1 \right].
  \] 
Since $K$ is constant, either $\cos \alpha_0 = 0$ or $\lambda=0$. These two cases immediately lead to the two cases described in the lemma.
\end{proof}

\begin{proposition}\label{prop:CGC-spheres}
There exists a rotationally invariant sphere $S_K$ with constant Gauss curvature $K$ in the Berger sphere $\mathbb{S}^3_\tau$ if and only if $K \ge  K_0$, where $K_0 = 4 - 3\tau^2$ if $\tau \leq 1$ and $K_0 = \frac{1}{\tau^2}$ if $\tau > 1$. Moreover, $S_K$ is unique up to ambient isometries and is embedded if and only if
   \begin{equation} \label{eq:condition-embeddedness}
   \frac{1}{\tau} \! \int_0^{r} \frac{\sqrt{\cos^2 x (1 \! - \! 2\lambda \sin^2 x)^2 \! - \! (1 \! - \! \lambda \sin^2 x)( 1 \! - \! K(1 \! - \! \lambda \sin^2 x)\sin^2 x )}}{\cos x \sqrt{1 \! - \! K(1 \! - \! \lambda \sin^2 x)\sin^2 x}} \mathrm{d}x < \pi,
  \end{equation}
where $r$ is determined by $\sin^2 r = \frac{1}{2\lambda}\bigl( 1 - \sqrt{1 - \frac{4}{K}\lambda} \bigr)$ if $\lambda \neq 0$ or by $\sin^2 r = \frac{1}{K}$ if $\lambda = 0$.
\end{proposition}

\begin{remark}
  The right hand sides of the expressions for $\sin^2 r$ are always contained in $[0,1]$ if $K \geq K_0$, so $r$ is well defined. 
\end{remark}

\begin{proof}
In the proof, we will use the function 
\begin{equation} \label{eq:def_J}
\mathcal F : [0,1] \times [-1,1] \to \mathbb R : (X,Y) \mapsto \frac{(1 - 2\lambda X)^2}{1 -\lambda X}(1-X)Y^2 + K(1 - \lambda X)X,
\end{equation}
which, when comparing to \eqref{eq:integral-formula}, satisfies $\mathcal E = \mathcal F(\sin^2x,\cos\alpha)$. 

Let $(x, y, \alpha)$ be a solution of \eqref{eq:ODE-Rotationally-invariant-surface} defined on a maximal interval $I \subset \mathbb{R}$ and assume that $\alpha$ is not constant, i.e., the solution is not a Clifford torus nor a great sphere in the round case (see Lemma~\ref{lm:constant-solutions}). The theory of \textsc{ode}s implies that a solution of \eqref{eq:ODE-Rotationally-invariant-surface} can be extended through values for which $(x, \alpha)$ lies in the interior of $[0,\frac{\pi}{2}] \times [0,\pi]$. Hence, the set $\{(x(s), \alpha(s)): s \in I\}$ is a connected component of a level curve $\{ (x, \alpha) \in [0,\frac{\pi}{2}]\!\times\![0,\pi] : \mathcal F(\sin^2x,\cos\alpha)=\mathcal E \}$, where $\mathcal E$ is the energy of the solution.

  For a solution to represent a sphere, it must touch the axis of revolution orthogonally twice, i.e., it must satisfy $\sin x(s) \to 0$ and $\cos\alpha(s) \to \pm 1$ if $s$ tends to the boundary points of $I$. As a consequence, $\mathcal E = 1$ by \eqref{eq:integral-formula} and the rest of the proof is based on analyzing the level $1$ curves of the function $\mathcal F$. 

  Let $\Gamma = \{(X,Y,\mathcal F(X,Y)) \in \mathbb R^3 : (X,Y) \in [0,1]\times[-1,1]\}$ be the graph of $\mathcal F$ and $\Pi = \{(X,Y,1) \in \mathbb R^3 : (X,Y) \in [0,1]\times[-1,1]\}$ a part of the plane at height $1$ in $\mathbb R^3$. A rotationally invariant \textsc{cgc} $K$ sphere exists if and only if there is a connected curve in $\Gamma \cap \Pi$ joining the points $(0,-1,1)$ and $(0,1,1)$. Observe that $\mathcal F(0, \pm 1) = 1$, so $(0,\pm 1, 1)\in \Gamma \cap \Pi$. We distinguish two cases depending on the sign of $\lambda = 1 - \tau^2$. 

\textbf{Case A.} $\lambda \geq 0$ (equivalently $\tau \leq 1$). If $K < 4 - 3\tau^2$ then
  \begin{equation}\label{eq:partials-J}
    \frac{\partial \mathcal F}{\partial X}(0,1) = K - (4 - 3\tau^2) < 0 \quad\text{and}\quad \frac{\partial \mathcal F}{\partial Y}(0,1) = 2.
  \end{equation} 
As a consequence, $\Gamma$ is below $\Pi$ in a neighborhood of $(0,1,1)$ and hence also in a neighborhood $(0,-1,1)$ because $\mathcal F$ is even in $Y$. Therefore, there can be no curve in $\Gamma \cap \Pi$ joining $(0, -1, 1)$ and $(0,1,1)$ (see left column in Figure~\ref{fig:contourplots-J-lambda-positive} for a contour plot of $\mathcal F$ in this case). 
  
  \begin{figure}[htpb]
    \centering

    \includegraphics[width=0.3\textwidth]{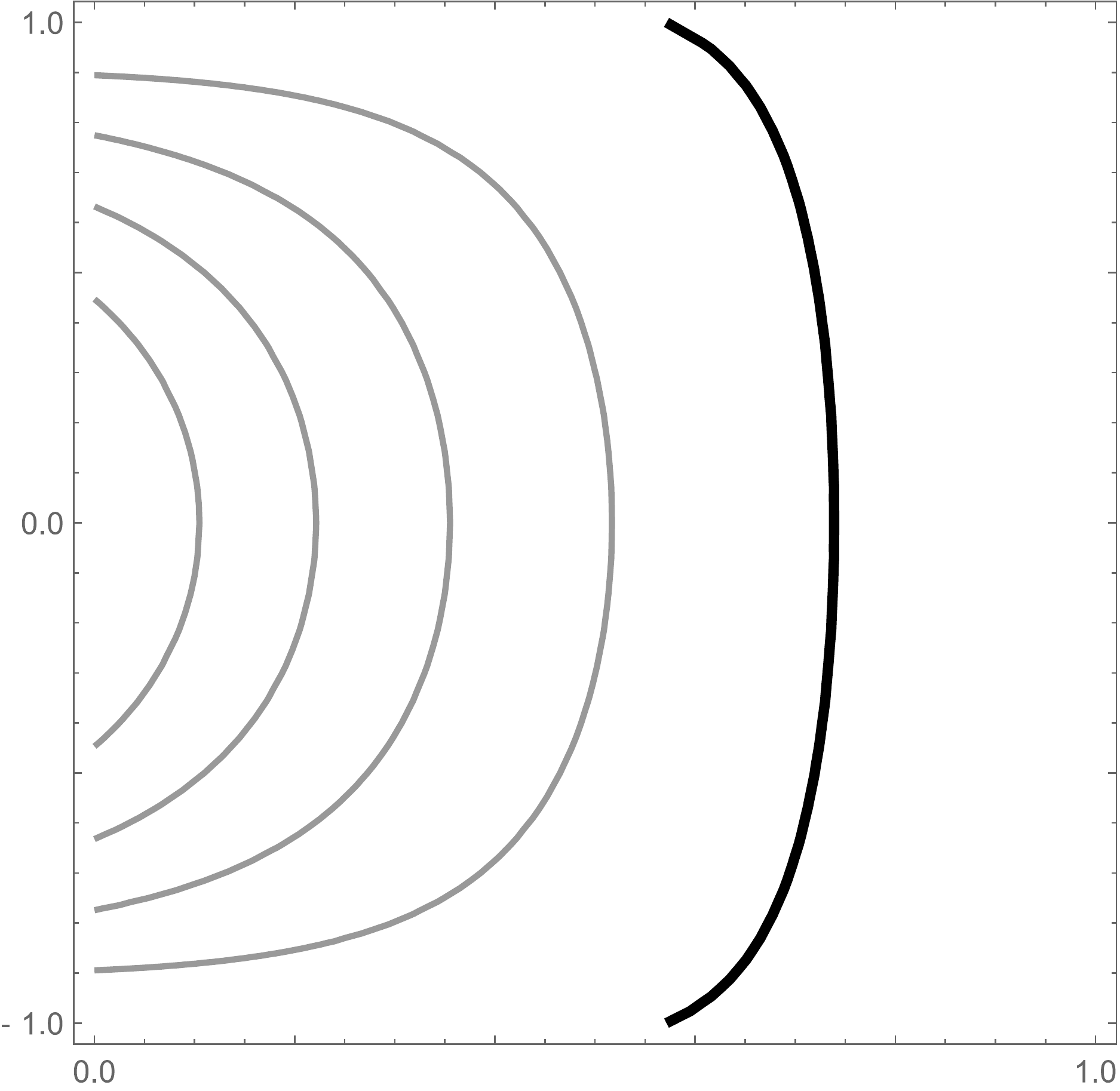}\quad
    \includegraphics[width=0.3\textwidth]{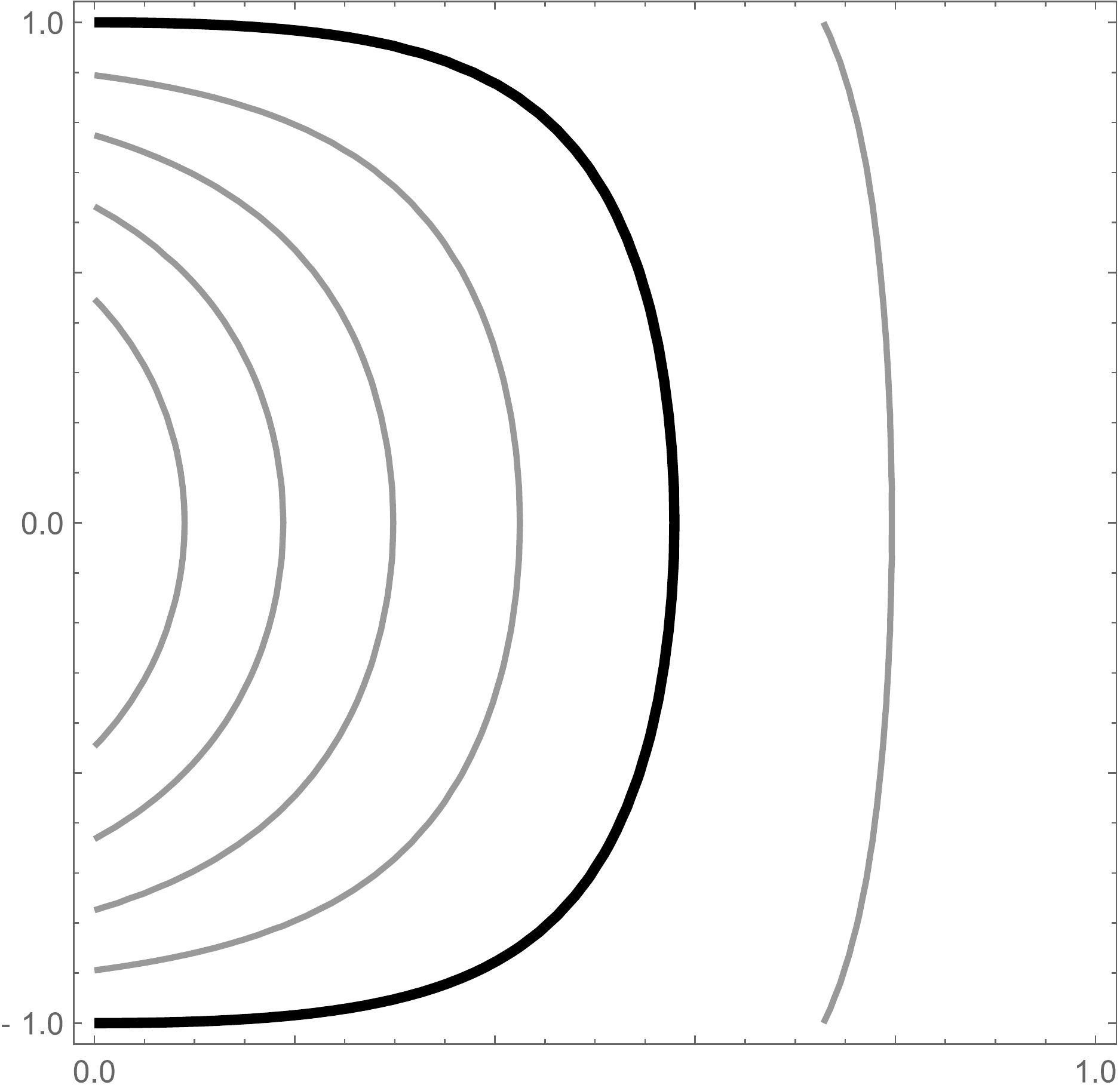}\quad
    \includegraphics[width=0.3\textwidth]{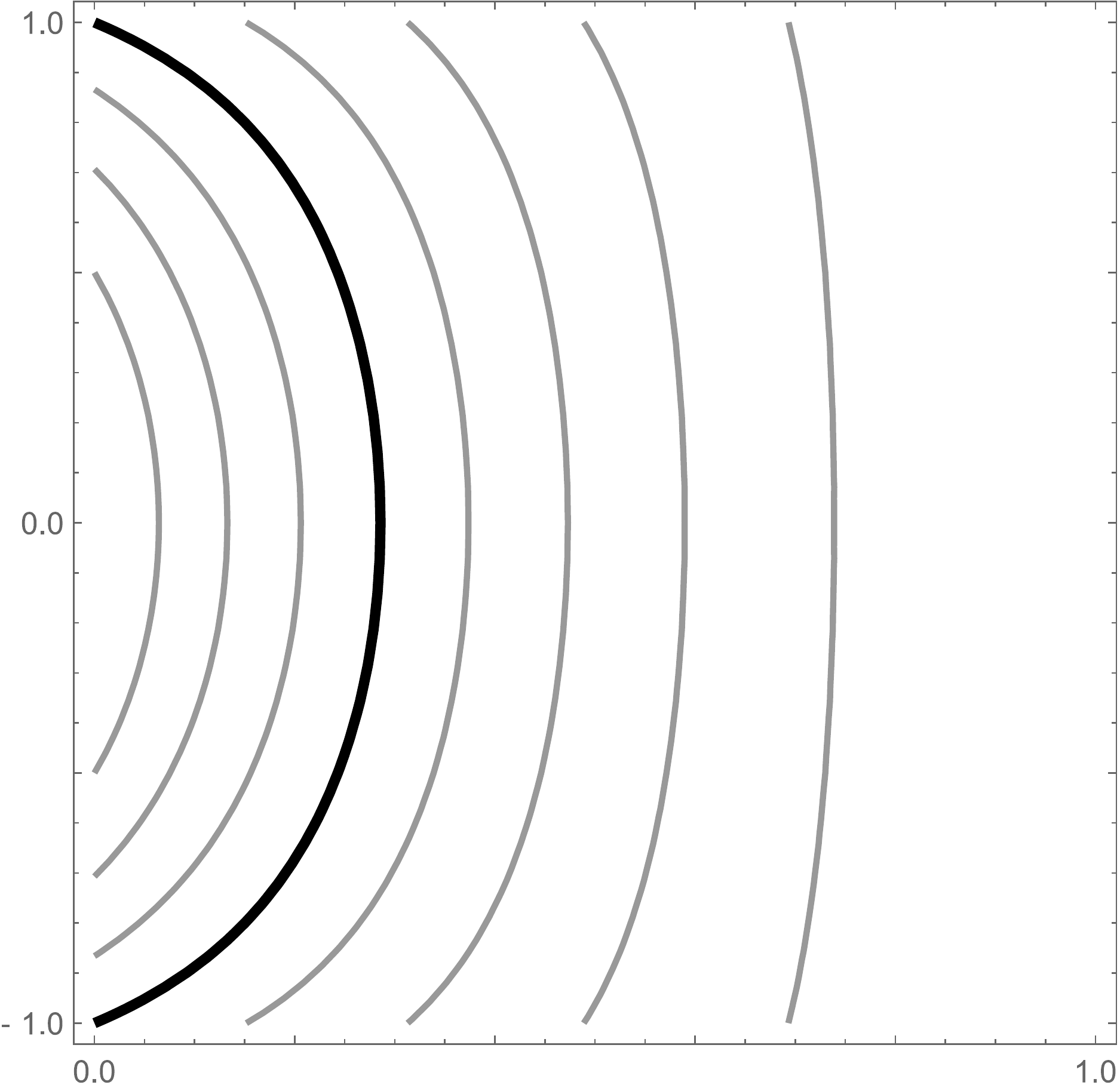}

    \includegraphics[width=0.3\textwidth]{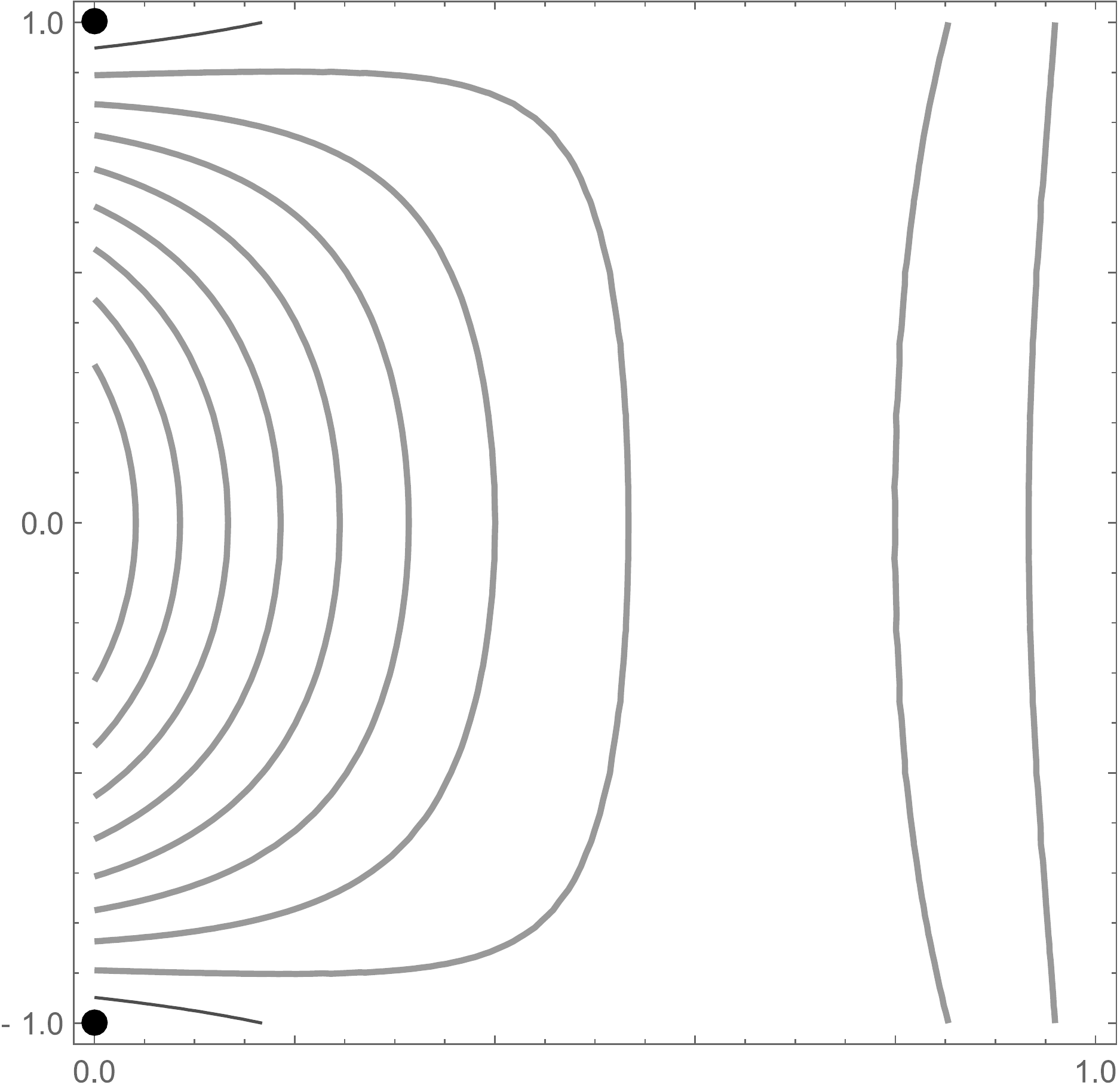}\quad
    \includegraphics[width=0.3\textwidth]{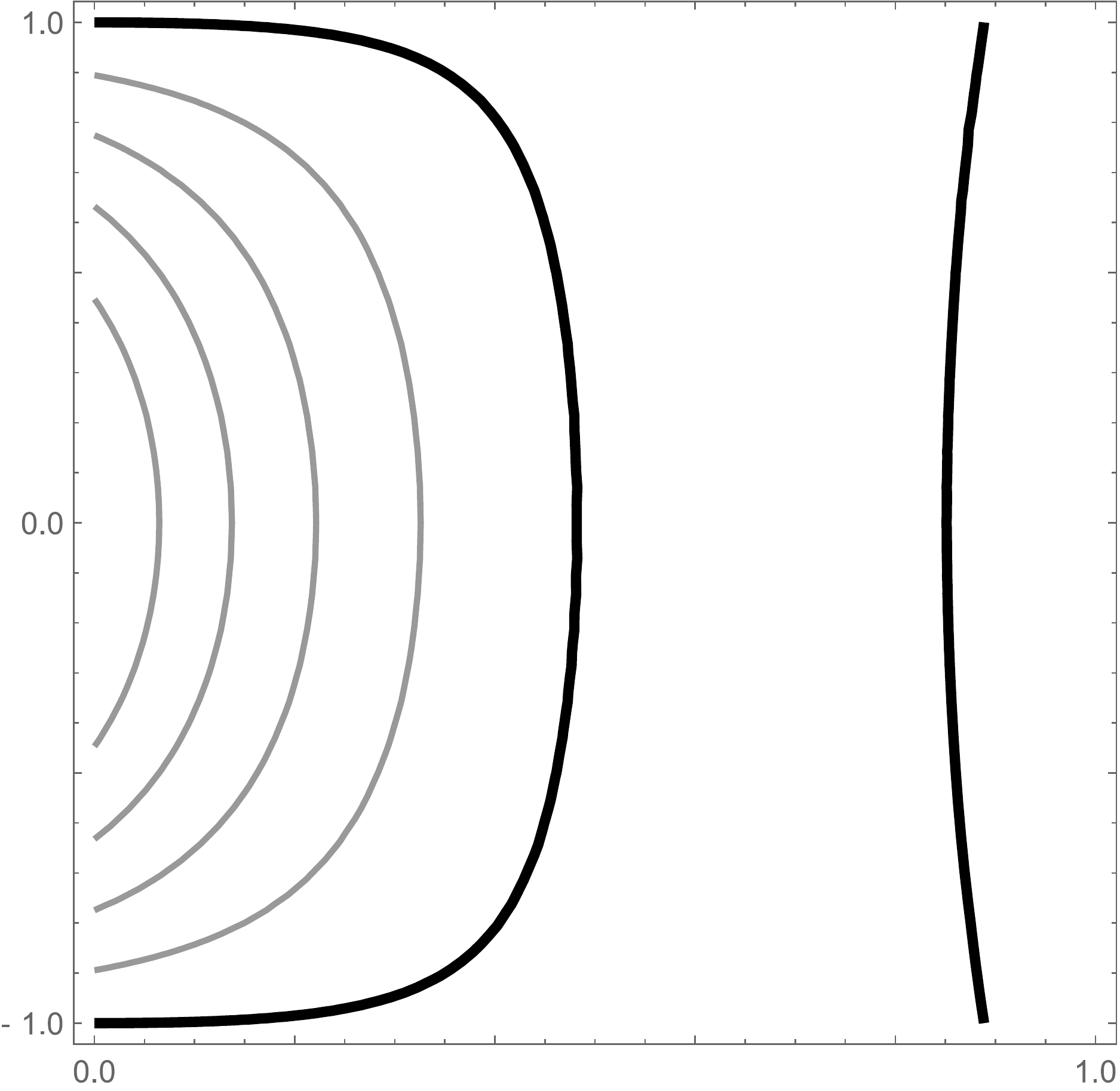}\quad
    \includegraphics[width=0.3\textwidth]{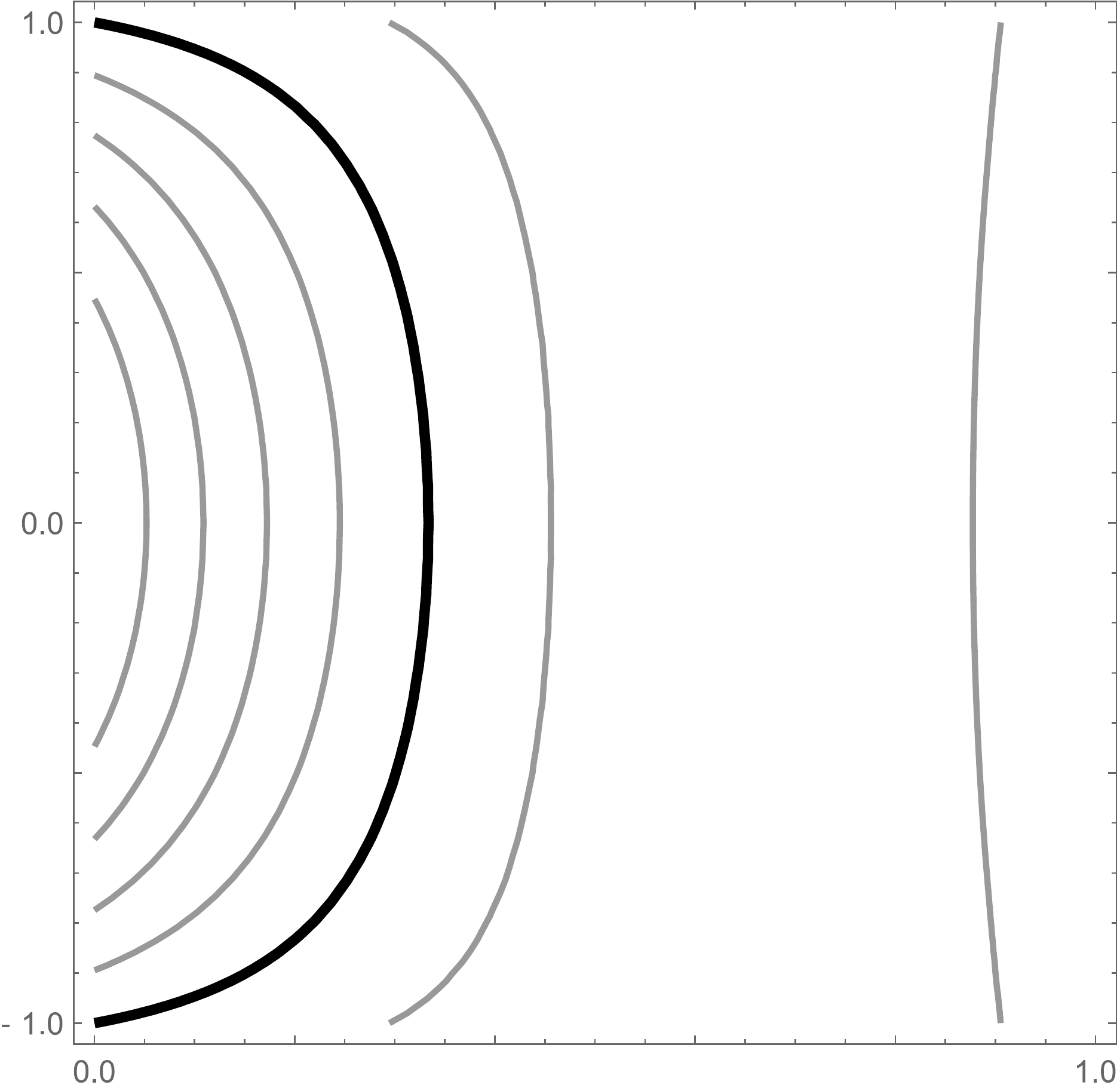}

    \caption{Contour plots of $\mathcal F$ for $K < 4 - 3\tau^2$ (left), $K = 4 - 3\tau^2$ (center) and $K > 4 - 3\tau^2$ (right) for $\tau = \frac{3}{4}$, so $\lambda = \frac{7}{16} < \frac{1}{2}$ (top row), and for $\tau = \frac{1}{2}$, so $\lambda = \frac{3}{4} > \frac{1}{2}$ (bottom row). The level $1$ curves are drawn in a bolder line.}
    \label{fig:contourplots-J-lambda-positive}
  \end{figure}

Let us now assume that $K \ge 4 - 3\tau^2$, equivalently $K \geq 3\lambda + 1$. On the one hand, we have that $\mathcal F(0,Y) = Y^2 \le 1$ for $Y \in [-1,1]$. On the other hand:
\[
\begin{split}
  \mathcal F(X, \pm 1) &= \frac{(1 - 2\lambda X)^2}{1 -\lambda X}(1-X) + K(1 - \lambda X)X \\
              &\geq \frac{(1 - 2\lambda X)^2}{1 -\lambda X}(1-X) + (3\lambda + 1)(1 - \lambda X)X \\
              &= 1 + \frac{\lambda(1-\lambda) X^2}{1 - \lambda X} (2 - 3 \lambda X),\\
\end{split}
\] 
where we have used the hypothesis on the Gauss curvature for the inequality. Since we are in the case $0 \leq \lambda < 1$, we get that $\mathcal F(X, \pm 1) \geq 1$ if $2-3\lambda X \geq 0$. We consider two cases:
  \begin{enumerate}
    \item If $0 \leq \lambda \le \frac{1}{2}$ then $2-3\lambda X \geq 0$ for all $X \in [0, 1]$ and thus $\mathcal F(X, \pm1) \geq 1$ for all $X \in [0, 1]$. Moreover, it is easy to check that $\mathcal F(1,Y) = K(1-\lambda) \ge 1$ because we are assuming $K\geq 3\lambda + 1$. The last two pictures in the top row of Figure~\ref{fig:contourplots-J-lambda-positive} are contour plots of $\mathcal F$ in this case.
    \item If $\frac{1}{2} < \lambda < 1$, then $2-3\lambda X \geq 0$ for all $X \in [0, \frac{1}{2\lambda}]$ and hence $\mathcal F(X, \pm1) \geq 1$ for all $X \in [0, \frac{1}{2\lambda}]$. Moreover, it is easy to check that $\mathcal F(\frac{1}{2\lambda},Y) = \frac{K}{4\lambda} \ge 1$ since we are assuming that $K \geq 3\lambda +1$. The last two pictures in the bottom row of Figure~\ref{fig:contourplots-J-lambda-positive} are contour plots of $\mathcal F$ in this case.
  \end{enumerate}
  In the first case, we get that $\Gamma$ is above $\Pi$ along three of its boundary curves and below along the fourth. In the second case, a similar argument can be carried out in the rectangle $[0, \frac{1}{2\lambda}]\times[-1,1]$. As a consequence, there exists a curve in $\Gamma \cap \Pi$ joining $(0,-1,1)$ and $(0,1,1)$.

\textbf{Case B.} $\lambda < 0$ (equivalently $\tau > 1$). In this case, if $K < \frac{1}{\tau^2}$ then $\mathcal F(X,0) = K(1-\lambda X)X \le K(1-\lambda) < 1$ for all $X \in [0,1]$. Hence, there can be no curve in $\Gamma \cap \Pi$ joining the points $(0,-1,1)$ and $(0,1,1)$ (see the first picture of Figure~\ref{fig:contourplots-J-lambda-negative} for a contour plot of $\mathcal F$ in this case). 

  \begin{figure}[htpb]
    \centering
    \includegraphics[width=0.3\textwidth]{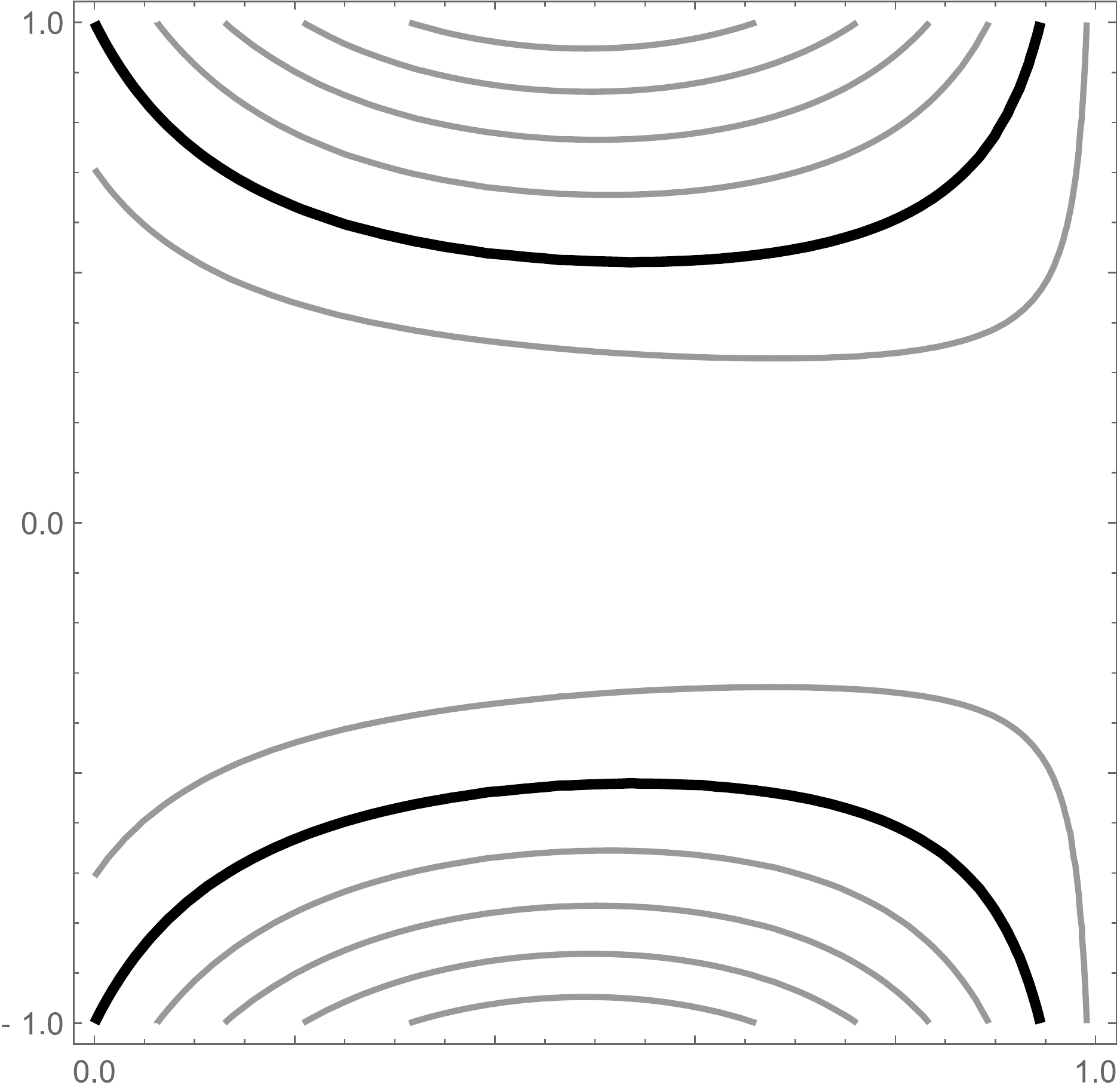}\quad
    \includegraphics[width=0.3\textwidth]{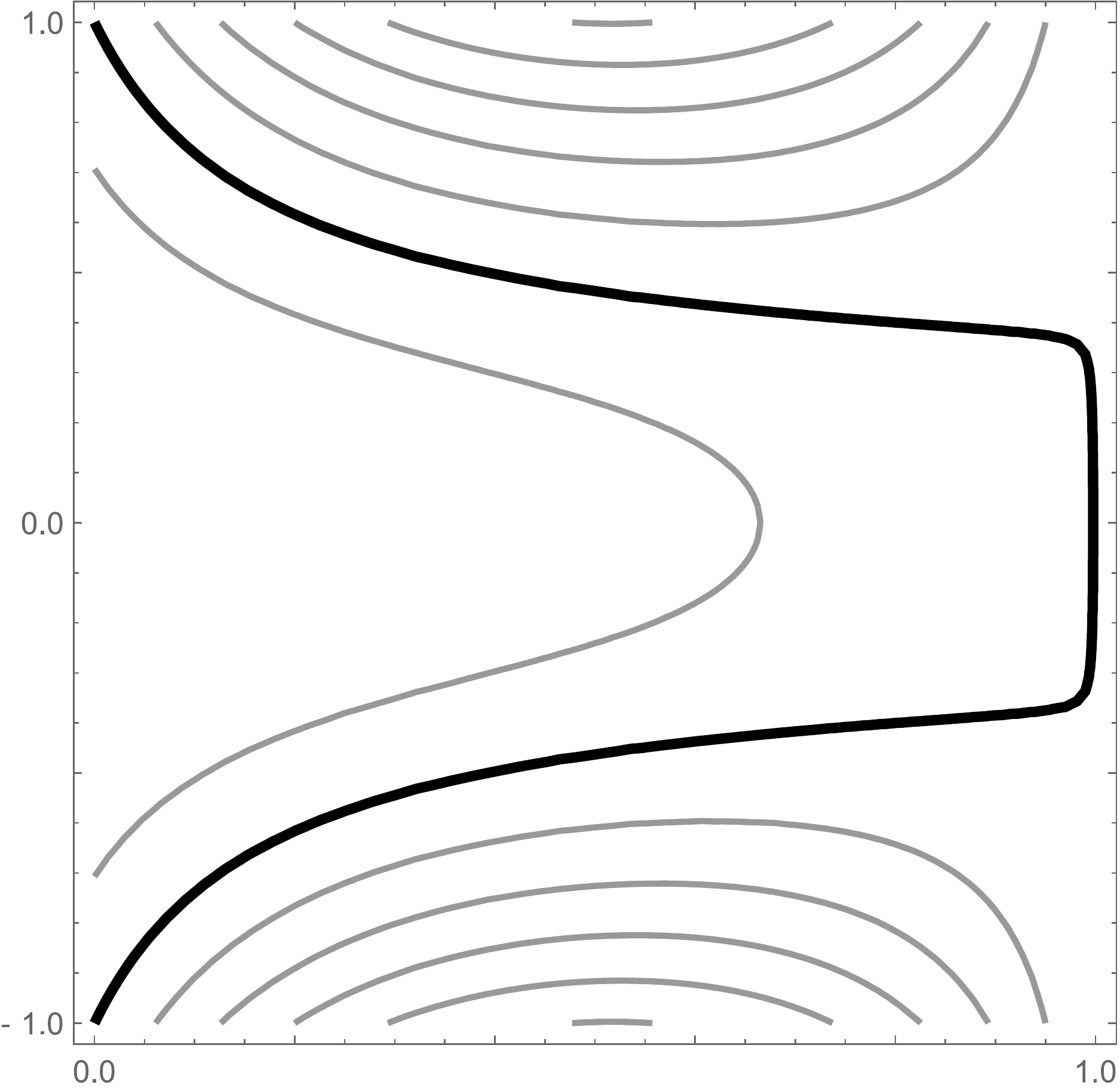}\quad
    \includegraphics[width=0.3\textwidth]{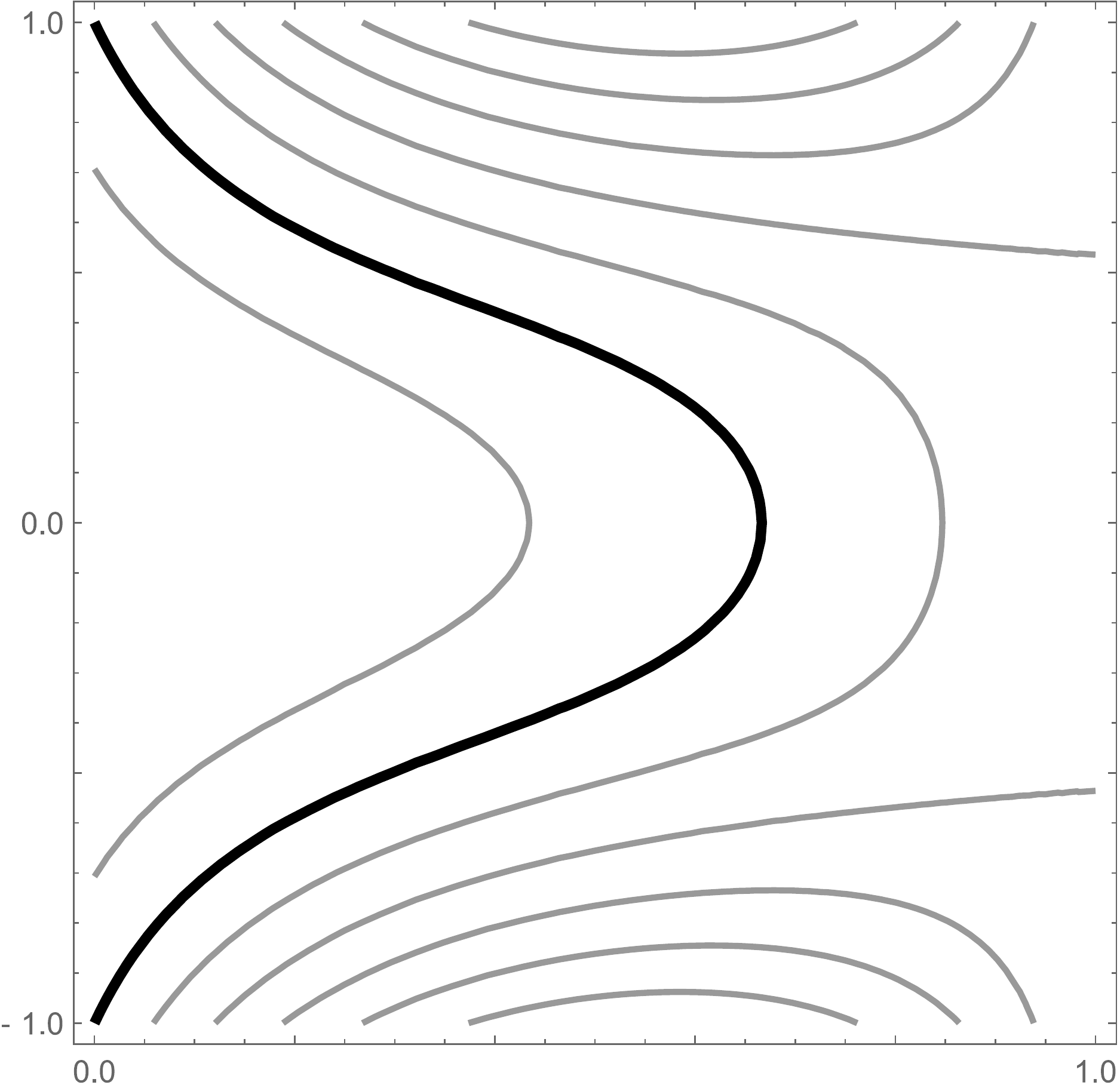}

    \caption{Contour plots of $\mathcal F$ for $K < \frac{1}{\tau^2}$ (left), $K = \frac{1}{\tau^2}$ (center) and $K > \frac{1}{\tau^2}$ (right) for $\tau = 2$, so $\lambda = -3$. The level $1$ curves are drawn in a bolder line.}
    \label{fig:contourplots-J-lambda-negative}
  \end{figure}

Assume now that $K \geq \frac{1}{\tau^2}$, that is, $K \geq \frac{1}{1-\lambda}$. On the one hand, $\mathcal F(1,Y) = K(1-\lambda) \geq 1$ and $\mathcal F(0,Y) = Y^2 \leq 1$ for all $Y  \in [-1,1]$. On the other hand,
\[
  \begin{split}
    \mathcal F(X, \pm 1) & = \frac{(1-2\lambda X)^2}{1 - \lambda X}(1-X) + K(1 - \lambda X)X \\
    &\geq  \frac{(1-2\lambda X)^2}{1 - \lambda X}(1-X) + \frac{1}{1-\lambda}(1 - \lambda X)X \\
    &= 1 - \frac{\lambda X(1-X)}{(1-\lambda)(1-\lambda X)} \bigl(\lambda(4\lambda - 3)X + 2 - 3\lambda \bigr) \geq 1 \\
  \end{split}
\] 
for all $X \in [0,1]$, where we have used the hypothesis on the Gauss curvature in the first inequality. Hence, there exists a curve in $\Gamma \cap \Pi$ joining $(0,-1,1)$ and $(0,1,1)$ (see the last two pictures in Figure~\ref{fig:contourplots-J-lambda-negative} for contour plots of $\mathcal F$ in this case).

We finally show that the solution corresponding to $\mathcal E = 1$ is a sphere and we will compute its \emph{horizontal radius} $r(\tau,K)$, i.e., the maximum distance to the axis of revolution, as well as its \emph{vertical radius} $h(\tau,K)$, i.e., half the maximum distance in the direction of the fibers for any two points in the sphere (see Figure~\ref{fig:sketch-sphere}).

Firstly, we have already proved that for $K \ge K_0$ (with $K_0$ as in the formulation of Proposition \ref{prop:CGC-spheres}), there exists a connected level $1$ curve of $\mathcal F$ joining the points $(0,1)$ to $(0,-1)$. Moreover, since both points are not critical for $\mathcal F$ (see~\eqref{eq:partials-J}), the level curve is unique. Hence, the solution associated to such level curve is unique up to a translation in $y$ (see Remark~\ref{rmk:symmetries-ODE-system}.(i)) and the corresponding profile curve $\gamma$ is defined on a compact interval $I = [0, T]$ such that $x(0) = 0$, $\alpha(0) = 0$, and  $x(T) = 0$, $\alpha(T) = \pi$. Since translation in $y$ corresponds to vertical translation along the fibers (see Remark~\ref{rmk:symmetries-ODE-system}.(i)), which are isometries of the Berger sphere, the profile curve and so the generated surface is unique up to ambient isometries.
By continuity of $\alpha$ there is an $s_1 \in ]0, T[$ such that $\alpha(s_1) = \frac{\pi}{2}$ and so $x'(s_1) = 0$ by~\eqref{eq:ODE-Rotationally-invariant-surface}. By Remark~\ref{rmk:symmetries-ODE-system}(v), the solution is symmetric with respect to the line $y = y(s_1)$. In that case, $x(2s_1) = 0$ and $\alpha(2s_1) = \pi$ so $2s_1 = T$. Moreover, we can assume by a vertical translation as in Remark~\ref{rmk:symmetries-ODE-system}(i) that $y(\frac{T}{2}) = 0$ (see Figure~\ref{fig:sketch-sphere}). 

  Now, $x'(s) > 0$ for $s \in [0, \frac{T}{2}[$ so we can write $y(s)$ as a function of $x(s)$. From \eqref{eq:ODE-Rotationally-invariant-surface} we deduce that
  \begin{equation}\label{eq:derivative-y-x}
    \frac{\mathrm{d}y}{\mathrm{d}x} = \frac{1}{\tau} \frac{\sqrt{1 - \lambda \sin^2 x}}{\cos x} \tan \alpha.
  \end{equation}
  Since $\alpha(s) \in [0, \frac{\pi}{2}]$ for $s \in [0, \frac{T}{2}]$, this implies that $y(x)$ is strictly increasing on $[0, x(\frac{T}{2})]$ (see Figure~\ref{fig:sketch-sphere}).

  \begin{figure}[htbp]
    \centering
    \includegraphics{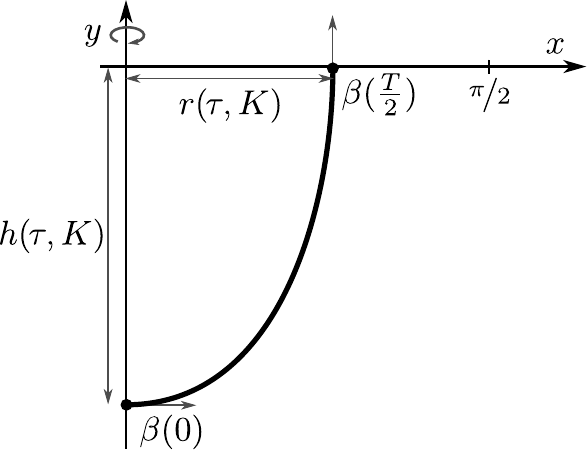}
    \caption{Sketch of the curve $\beta(s) = \bigl(x(s), y(s)\bigr)$ in the interval $[0, \frac{T}{2}]$ for a solution $(x, y, \alpha)$ with energy $\mathcal E = 1$ and curvature $K \geq K_0$. By the symmetries of the system~\eqref{eq:ODE-Rotationally-invariant-surface} (see Remark~\ref{rmk:symmetries-ODE-system}) the curve is symmetric with respect to the $x$-axis. The function $y(x)$ is strictly increasing in that interval by~\eqref{eq:derivative-y-x}. The profile curve $\gamma(s) = (e^{iy(s)} \cos x(s), \sin x(s))$ is embedded if and only if  $h(\tau, K) < \pi$.}
    \label{fig:sketch-sphere}
  \end{figure}

  As a consequence, by evaluating~\eqref{eq:integral-formula} at $s = \frac{T}{2}$ we deduce that the horizontal radius $r(\tau, K) = x(\frac{T}{2})$ satisfies
  \begin{equation}\label{eq:value-x-equator-sphere}
    \sin^2 r(\tau, K) = \sin^2 x(\tfrac{T}{2}) = \tfrac{1}{2\lambda}\left(1 - \sqrt{1 - \tfrac{4}{K}\lambda}\right),
  \end{equation}
  and that the vertical radius is
  \begin{equation}\label{eq:vertical-radius} 
    h(\tau, K) = y(r) - y(0) = \int_0^{r} \frac{\mathrm{d}y}{\mathrm{d}x} \,\mathrm{d}x = \frac{1}{\tau} \int_0^{r} \frac{\sqrt{1 - \lambda \sin^2 x}}{\cos x} \tan \alpha \,\mathrm{d}x.
  \end{equation}
  Computing $\tan \alpha$ from~\eqref{eq:integral-formula} gives the expression on the left hand side of \eqref{eq:condition-embeddedness}. The resulting surface will be embedded if and only if $h(\tau, K) < \pi$ because otherwise the image of the profile curve $\gamma(s) = (e^{iy(s)}\cos x(s), \, \sin x(s))$ will not be embedded in $H$ (see Figure~\ref{fig:sketch-sphere}).
\end{proof}

\begin{figure}[htbp]
  \centering
  \includegraphics{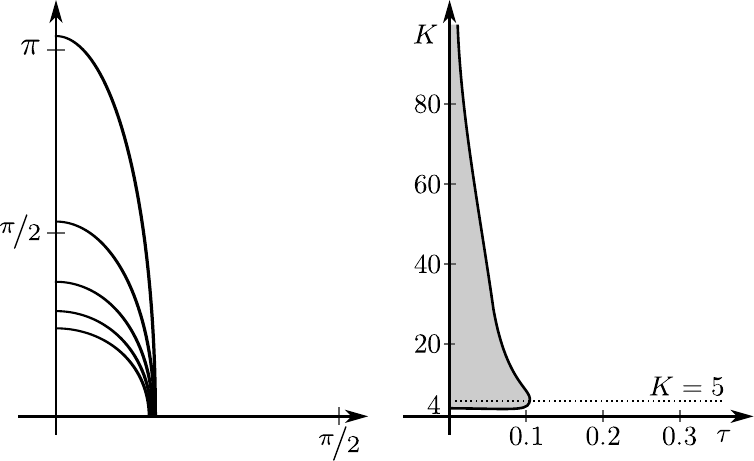}
  \caption{Left: numerical computation of the curves $\beta(s) = (x(s), y(s))$ (see Figure~\ref{fig:sketch-sphere}) of constant Gauss curvature $K = 5$ spheres in $\mathbb{S}^3_\tau$ for $\tau \in \{0.1, 0.2, 0.3, 0.4, 0.5\}$. Notice how the sphere corresponding to $\tau = 0.1$ is not embedded since its \emph{vertical radius} $h(0.1,5) > \pi$. Right: non-embedded region with boundary the curve given implicitly by $h(\tau,K) = \pi$.}
  \label{fig:region-non-embeddedness}
\end{figure}

\begin{remark}
  It is easy to check that the vertical radius (see~\eqref{eq:vertical-radius}) satisfies $\lim_{\tau \to 0} h(\tau,K) = +\infty$ so, by continuity, if $\tau$ is sufficiently small then there will be non-embedded constant Gauss curvature spheres in $\mathbb{S}^3_\tau$ (see Figure~\ref{fig:region-non-embeddedness}).
\end{remark}

\section{Proof of the main result}

We now prove the main theorem, which ensures that the only complete \textsc{cgc} rotationally invariant surfaces are the Clifford tori and the spheres described in Proposition~\ref{prop:CGC-spheres}.

\begin{proof}[Proof of Theorem~\ref{thm:main-theorem}]
  It is clear that the tori and the spheres described in Lemma~\ref{lm:constant-solutions} and Proposition~\ref{prop:CGC-spheres} are complete rotationally invariant surfaces in $\mathbb S^3_\tau$. This proves the existence part of Theorem~\ref{thm:main-theorem}.  

  To prove the uniqueness part, assume that $S$ is a complete rotationally invariant surface with constant curvature $K$ in $\mathbb S^3_\tau$ and denote its profile curve by $\gamma: I \subset \mathbb{R} \to H \subset \mathbb{S}^2(1)$. If the function $\alpha$ associated to the profile curve (see~\ref{eq:ODE-Rotationally-invariant-surface}) is constant then $S$ is a Clifford torus in $\mathbb{S}^3_\tau$ or a totally geodesic sphere in $\mathbb{S}^3_1$ by Lemma~\ref{lm:constant-solutions}. Let assume for the rest of the proof that $\alpha$ is not constant. From the discussion at the end of Remark~\ref{rmk:symmetries-ODE-system}, either $I = \mathbb{R}$ or the only possible values of $(\sin^2 x, \cos \alpha)$ at the boundary of $I$ are $(0, \pm 1)$ or $(1, a)$ for some $a \in [-1,1]$. 
  
Let us first assume that $K > 0$. Then (the universal cover of) $S$ must be a \textsc{cgc} $K$ sphere and we know from Proposition~\ref{prop:CGC-spheres} that $K \geq K_0$ for the value of $K_0$ given in the formulation of the theorem and that the corresponding rotationally invariant sphere is unique.

Finally, we assume that $K \le 0$. It suffices to obtain a contradiction in this case. The function $\mathcal F$ defined by \eqref{eq:def_J} satisfies
\begin{equation}\label{eq:upper-bound-J-K-negative}
  \mathcal F(X,0) = K(1-\lambda X) \leq 0 < 1
\end{equation}
for all $ X \in [0, 1]$. Moreover, 
  \begin{equation}\label{eq:lower-bound-J-K-negative}
    \mathcal F(X,Y) \geq K(1-\lambda X)X \geq 
    \begin{cases}
      K(1-\lambda) & \text{if } \lambda \le \frac{1}{2},  \\
      \frac{K}{4\lambda} & \text{if } \lambda > \frac{1}{2} 
    \end{cases}
  \end{equation}
for all $(X,Y) \in [0,1] \times [-1,1]$ and the minimum value is only attained at the points $(1, Y)$ if $\lambda \leq \frac{1}{2}$ and at the points $(\frac{1}{2\lambda}, Y)$ if $\lambda > \frac{1}{2}$, which are the only interior critical points of $\mathcal F$. Hence, the energy $\mathcal E = \mathcal F(\sin^2x, \cos\alpha)$ takes values in a compact interval. 
  
If $\gamma$ touches the axis of revolution, i.e., if there exists an $s_0 \in I$ with $x(s_0) = 0$, it has to be orthogonally because of completeness. Hence, $\cos \alpha(s_0) = \pm 1$ and $\mathcal E = 1$ by~\eqref{eq:integral-formula}. But, thanks to~\eqref{eq:upper-bound-J-K-negative}, the level $1$ curve of $\mathcal F$ starting at $(0, 1)$ (resp.\ $(0, -1)$) cannot end on the line segment $Y = - 1$ (resp.\ $Y = 1$). Hence, either it ends at some point $(X_0, 1)$ (resp.\ $(X_0, -1)$) with  $X_0 \in ]0, 1[$ or at $(1, Y_0)$. The former will produce a non-complete surface and the latter is impossible since  $\mathcal F(1, Y) = K(1-\lambda) \le 0$.

On the other hand, if $\gamma$ contains the north pole, i.e., if there is an $s_0 \in I$ with $\sin x(s_0) = 1$, then $\mathcal E = K(1-\lambda)$. This means that $\mathcal F(\sin^2x(s),\cos\alpha(s)) = K(1-\lambda)$ for all $s \in I$. Remark that this is precisely the value that $\mathcal F$ takes on the line segment $\{1\}\times [-1,1]$. Since $\mathcal F$ has no critical points on this line segment, there can be no level $K(1-\lambda)$ curve of $\mathcal F$ intersecting this line segment, other than the line segment itself. We conclude that $x(s) = \frac{\pi}{2}$ constant.  This means that the profile curve $\gamma$ degenerates to a point, namely to the north pole in $H \subset \mathbb S^2(1)$ and no surface is produced. As a consequence, we can also assume that $\gamma$ does not contain the north pole.

As a consequence of the previous paragraphs, the interval $I = \mathbb{R}$, that is, $(\sin^2x(s),\cos\alpha(s))$ is defined for all $s$ and lies in the interior of $[0,1]\times[-1,1]$. Since it is a compact regular curve, it must be a closed Jordan curve. Because $\mathcal F$ is constant on the curve, there must be a critical point of $\mathcal F$ in the interior of the domain bounded by it. If $\lambda \leq \frac{1}{2}$, this is impossible since $\mathcal F$ has no interior critical points. If $\lambda > \frac{1}{2}$, then the critical point inside the domain takes the form $(\frac{1}{2\lambda}, Y_0)$ for some $Y_0 \in ]-1,1[$ (see~\eqref{eq:lower-bound-J-K-negative} and the paragraph below). The line segment $\{\frac{1}{2\lambda}\} \times [-1,1]$ and the Jordan curve must hence intersect. As a consequence, for all $s \in I$, $\mathcal F(\sin^2 x(s), \cos \alpha(s)) = \mathcal F(\frac{1}{2\lambda}, Y_0) = \frac{K}{4\lambda}$, which is the minimum value of $\mathcal F$ (see~\eqref{eq:lower-bound-J-K-negative}). But this value is only attained on  $\{\frac{1}{2\lambda}\}\times[-1,1]$ so, $x$ must be constant and the associated surface is a Clifford torus by Lemma~\ref{lm:constant-solutions}.
\end{proof}

Finally, by combining Pogorelov's uniqueness result and the description of the profile curve of the spheres $S_K$ obtained above (see Figure~\ref{fig:sketch-sphere}), we show that they are the only \textsc{cgc} spheres for $K > 4 - 3\tau^2$ if $\tau^2 < 1$ and $K > \tau^2$ if $\tau^2 > 1$ (see the striped region in Figure~\ref{fig:non-existence-region-compact-cgc-surfaces-Berger}).

\begin{proof}[Proof of Corollary~\ref{cor:uniqueness}]
Since the sectional curvature of $\mathbb S^3_\tau$ is bounded from above by $4 - 3\tau^2$ if $\tau < 1$ and by $\tau^2$ otherwise (see Section~\ref{sec:introduction}) , the result follows from~\cite[Theorems~1 and~2, p.~413--418]{Pogorelov1973} once we show that, for every point $p\in \mathbb{S}^3_\tau$ and every plane $\Pi \subset T_p \mathbb{S}^3_\tau$, there exists a rotationally invariant sphere that passes through $p$ and is tangent to $\Pi$. But this is a consequence of $\mathbb{S}^3_\tau$ being an homogeneous manifold and the shape of the rotationally invariant spheres (see Figure~\ref{fig:sketch-sphere}).

More precisely, let $\Pi \subset T_{p} \mathbb{S}^3_\tau$ be any tangent plane at $p$ and $\eta$ a unit normal vector field to $\Pi$. Let us denote by $\nu_0 = g_{\tau}(\eta, \xi) \in [-1, 1]$. Let  $\gamma(s) = \bigl(e^{iy(s)}\cos x(s), \sin x(s)
\bigr)$ be the profile curve of a rotationally invariant \textsc{cgc} sphere $S_K$ (see Figure~\ref{fig:sketch-sphere}). The functions $x$ and $y$ are strictly increasing in $s \in [0, \frac{T}{2}]$ (see~\eqref{eq:derivative-y-x} and the arguments before that equation). Hence, the function $\nu(s) = g_{\tau}(N_{\gamma(s)}, \xi)$, where  $N$ is the exterior normal vector field to $S_K$ is strictly monotone for $s \in [0, \frac{T}{2}]$. Moreover, $\nu(0) = -1$ and $\nu(\frac{T}{2}) = 0$. Hence, $\nu([0, \frac{T}{2}]) = [-1, 0]$ and, by symmetry, $\nu([\frac{T}{2},0]) = [0, 1]$. Hence, there exists a unique $s_0 \in [0, T]$ such that $\nu(s_0) = \nu_0$. Translate $S_K$ from  $\gamma(s_0)$ to $p$ so $N_p$ and $\eta$ lie in the same tangent space  $T_{p}\mathbb{S}^3_\tau$ and have the same orthogonal projection to $\xi_{p}$.  By the rotational invariance of the sphere, $S_K$ will be tangent to  $\Pi$ after a rotation of angle $\theta$, with $\cos\theta = g_{\tau}(N_{\gamma(s_0)}, \eta)$.
\end{proof}

\end{document}